\documentclass[12pt,a4paper]{amsart}
\usepackage{amsmath,amsthm,amssymb,latexsym,a4wide,epic,eepic,bbm}

\usepackage{color}

\usepackage{ulem}
\renewcommand{\em}{\it}

\newtheorem{theorem}{Theorem}[section]
\newtheorem{lemma}[theorem]{Lemma}
\newtheorem{corollary}[theorem]{Corollary}
\newtheorem{proposition}[theorem]{Proposition}

\newtheorem{example}[theorem]{Example}
\newtheorem{remark}[theorem]{Remark}

\makeatletter
\newcommand*{\drightleftarrow}[2]{\mathrel{
  \settowidth{\@tempdima}{$\scriptstyle#1$}
  \settowidth{\@tempdimb}{$\scriptstyle#2$}
  \ifdim\@tempdimb>\@tempdima \@tempdima=\@tempdimb\fi
  \mathop{\vcenter{
    \offinterlineskip\ialign{\hbox to\dimexpr\@tempdima+2em{##}\cr
    \rightarrowfill\cr\noalign{\kern.3ex}
    \leftarrowfill\cr}}}\limits^{\!#1}_{\!#2}}}

\newcommand*{\drightarrow}[2]{\mathrel{
  \settowidth{\@tempdima}{$\scriptstyle#1$}
  \settowidth{\@tempdimb}{$\scriptstyle#2$}
  \ifdim\@tempdimb>\@tempdima \@tempdima=\@tempdimb\fi
  \mathop{\vcenter{
    \offinterlineskip\ialign{\hbox to\dimexpr\@tempdima+2em{##}\cr
    \rightarrowfill\cr\noalign{\kern.3ex}
    \rightarrowfill\cr}}}\limits^{\!#1}_{\!#2}}}
\makeatother

\usepackage[active]{srcltx}

\usepackage{enumerate}
\numberwithin{equation}{section}

\title{The max-plus algebra of exponent matrices of tiled orders}

\author[M. Dokuchaev]{Mikhailo Dokuchaev}
\address{Departamento de Matematica Univ. de S\~ao Paulo\\ Caixa
Postal 66281, S\~ao Paulo, SP 05314-970, Brazil}
\email{dokucha@gmail.com}

\author[V. V. Kirichenko]{Vladimir V. Kirichenko}
\address{Faculty of Mechanics and Mathematics, Taras Shevchenko
National Univ. of Kyiv, Volodymyrska Str., 64, 01033 Kyiv, Ukraine}
\email{vv.kirichenko@gmail.com}

\author[G. Kudryavtseva]{Ganna Kudryavtseva}
\address{
Faculty of Civil and Geodetic Engineering, University of Ljubljana, \\ Jamova cesta~2, SI-1000 Ljubljana, Slovenia}
\email{ganna.kudryavtseva@fgg.uni-lj.si}

\author[M. Plakhotnyk ]{Makar Plakhotnyk}
\address{Departamento de Matematica Univ. de S\~ao Paulo\\ Caixa
Postal 66281, S\~ao Paulo, SP 05314-970, Brazil}
\email{makar.plakhotnyk@gmail.com}


\begin{document}

\begin{abstract}
An exponent matrix is an $n\times n$ matrix $A=(a_{ij})$ over
${\mathbb N}^0$ satisfying (1) $a_{ii}=0$ for all $i=1,\ldots, n$
and (2) $a_{ij}+a_{jk}\geq a_{ik}$ for all pairwise distinct
$i,j,k\in\{1,\dots, n\}$. In the present paper we study the set
${\mathcal E}_n$ of all non-negative $n\times n$
exponent matrices as an algebra with the operations $\oplus$ of
component-wise maximum and $\odot$ of component-wise addition. We
provide a basis of the algebra $({\mathcal E}_n, \oplus, \odot,0)$
and give a row and  a column decompositions of a matrix
$A\in {\mathcal E}_n$ with respect to this basis.
This structure result determines all $n\times n$ tiled orders over
a fixed discrete valuation ring.  We also study automorphisms of
${\mathcal E}_n$ with respect to each of the operations $\oplus$
and $\odot$ and prove that   ${\rm
Aut}(\mathcal{E}_n,\, \odot ) = {\rm
Aut}(\mathcal{E}_n,\, \oplus ) =  {\rm
Aut}(\mathcal{E}_n,\, \odot ,\oplus  ,0) \simeq  {\mathcal{S}}_n \times C_2,$
$n>2.$

\end{abstract}

\maketitle

\section{Introduction}

Orders over domains is a classical object of study, originated by Dedekind's  ideal theory of maximal orders in algebraic number fields.  Apart from  their own interest as a ``noncommutative arithmetic'', orders have   also great importance to the theory of integral representations and to integer matrices \cite{Reiner}. Orders of tiled form appeared as structural ingredients in the study of hereditary orders (see \cite{Harada-1963(2)} or \cite{Reiner}),  Bass orders \cite{DrozdKirRoiter} and, more ge\-ne\-ral\-ly, they are used in the context of quasi-Basss orders in  \cite{DrozdKir1972}. The latter two references  witness the  essential role of tiled orders in the theory of orders of finite representation type, whereas  their importance for the investigation of global dimension stems from Tarsy's paper \cite{Tarsy-1970}.

Various aspects of tiled orders have been extensively studied in the literature. These include homological aspects \cite{ jate1, jate2,Fujita-1990,Fujita-2001,JanOde-1997,KirkKuz-1989, KoenigEtc,RogKirKhiZhu-2001,Rump2},
representation theory \cite{Rump,Simson1,Simson2,Simson3,Zav-Kir1977}, structure  \cite{FujitaSakai,PierceEtc,YangYu,WeideRoggen-1984,Zav-Kir}, $K$-theory \cite{Keating,PengGuo} and others.  In addition,  tiled orders turned out to be useful to prove  Krull-Remak-Schmidt-Azumaya type theorems in additive categories
 \cite{DArnold} and, more recently, a strong connection between  cluster categories and Cohen-Macaulay representation theory of some tiled orders  was established in  \cite{DemonetLuo}.

Notice that   the term ``tiled'' for rings  was used first time by
R. B.  Tarsy \cite{Tarsy-1970} and independently  by D. Eisenbud
and J. C. Robson \cite{EisenRob}. Since then the term ``tiled
oder'' became well established in referring to matrix rings of
``tiled form'' over a domain, however, orders  with tiled
structure over non-commutative rings appeared already in
\cite{DrozdKir1972, DrozdKirRoiter,  Keating, Zav-Kir}, and the
more general concept of a tiled ring was defined in
\cite{DokKirPolMi}. Nevertheless, tiled orders sometimes appear in
the literature under other names, such as   {\em Schurian orders}
\cite{WeideRoggen-1984} or {\em monomial orders}  \cite{YangYu}.

The current paper is concerned with {\em exponent matrices} of tiled orders \cite[Chapter  14]{HGK-1} which play a crucial role in characterization of  these orders.  Exponent matrices  are  $n\times n$   matrices over non-negative integers satisfying:

\begin{enumerate}[(EM1)]
\vspace{-0.05cm}
\item $a_{ii}=0$ for all $i=1,\dots, n$.
\item $a_{ij}+a_{jk}\geq a_{ik}$ for all pairwise distinct $i,j,k\in\{1,\dots, n\}$.
\end{enumerate}
\vspace{-0.01cm}
Of course (EM2) is non-vacuous only starting with $n=3$. As the definition suggests, exponent matrices are objects with a strong combinatorial flavour. Throughout the paper, the set of all exponent $n\times n$ matrices over ${\mathbb N}^0 ={\mathbb N}\cup \{0\}$ is denoted by ${\mathcal E}_n$.

The main idea of the present paper is to look at ${\mathcal E}_n$ as an algebra with respect to operations of component-wise maximum, denoted by $\oplus $ and sometimes called the tropical sum,  and component-wise addition, denoted by $\odot$ and sometimes called the tropical product.
Most of usual axioms of an idempotent semiring hold in the algebra $({\mathcal E}_n, \oplus, \odot, 0)$ where $0$ denotes the zero matrix: both of the operations $\oplus$ and $\odot$ are associative and commutative,  $\oplus$ is idempotent and $\odot$ distributes over $\oplus$. Observe, however, that in our algebra the neutral elements for both of the operations coincide: this is the zero matrix $0$.

The equational theory of the algebra $({\mathbb N}^0, \odot, \oplus, 0)$ was studied in  \cite{AEI, AEI1}. According to J.-E. Pin \cite{Pin} the adjective ``tropical'', in relation to a max-plus (or a min-plus) algebra,  was coined  by Dominique Perrin in honor of the pioneering work of  Imre Simon (1943-2009),  a mathematician and computer scientist from University of S\~ao Paulo, who was first to use min-plus semirings in theoretical computer science.  Namely,
 these semirings are crucial ingredients of I. Simon's solution of some famous decidability problems on rational languages, treating them from the point of view of
 Burnside type questions \cite{Simon1978,Simon1988} (see also \cite{MandelSimon} and  \cite{Simon1994}).

In the current paper we give a basis for the  max-plus algebra  $({\mathcal{E}}_n, \odot, \oplus, 0)$ and also study the  symmetry of  ${\mathcal{E}}_n$ from various points of view.
We now describe this basis. Let $I\subseteq \{1,2,\dots, n\}$ be a  proper subset, which means that $1\leq |I|\leq n-1$. We let $I^c=\{1,2,\dots, n\}\setminus I$ be the complement of $I$. By $T_{I}=(t_{ij})$ we denote the matrix given by
$$
t_{ij}=\left\lbrace\begin{array}{ll} 1, & i\in I, j\in I^c;\\
0, & \text{otherwise.}\end{array}\right.
$$
Let ${\mathcal T}$ be the set of all matrices $T_I$ where $I\subseteq  \{1,2,\dots, n\}$ and $1\leq |I| \leq n-1$. We sometimes call the elements of ${\mathcal T}$ {\em blocks}. It is easy to see that ${\mathcal T}\subseteq  {\mathcal{E}}_n$. We can now state our structure result.

\begin{theorem}[Structure Theorem] \label{th:structure_th}The matrices $T_I$, where $I$ runs through the proper subsets of  the set $\{1,2,\dots, n\}$, form a basis of the algebra $({\mathcal{E}}_n, \odot, \oplus, 0)$. That is, any matrix $A\in {\mathcal{E}}_n$ can be written in the form
\begin{equation}\label{eq:anja}
A=B_1\odot\ldots \odot B_l \oplus \ldots \oplus C_1\odot\ldots \odot C_m,
\end{equation}
where all the matrices $B_1,\dots, C_m$ are blocks (as usual $\odot$ is performed prior to $\oplus$). Moreover, this basis is the only minimal basis of the algebra $({\mathcal{E}}_n, \odot, \oplus, 0)$.
\end{theorem}

Theorem~\ref{th:structure_th} is proved in Section~\ref{sec:proof_structure}.
Notice that this result 
gives a way to obtain all tiled orders over a fixed  discrete valuation ring  from a simply described set of exponent matrices (see~\cite[pp.~352-353]{HGK-1}). 
In Section~\ref{sec:aut_odot} we study the automorphisms of the
semigroup $(\mathcal{E}_n, \odot)$ and prove in
Theorem~\ref{th:aut_odot} that ${\mathrm{Aut}}(\mathcal{E}_n,
\odot) \simeq {\mathcal{S}}_n \times C_2$, if $n\geq 3,$ where
${\mathcal{S}}_n$ stands for the symmetric group on $n$ letters
and $C_2$ denotes the cyclic group of order $2$. In order to study the automorphisms of $({\mathcal{E}}_n, \oplus )$ we need some technical preparation which is done
in the first part of Section~\ref{sec:aut_oplus}, 
considering {\em strict downsets} $A^{\Downarrow}=\{B\in {\mathcal
E}_n\colon B\lneq A\}$ of elements  $A\in {\mathcal E}_n. $   In
Theorem \ref{th:downset} we prove that a matrix $A\in {\mathcal
E}_n\setminus {\mathcal T}$ is uniquely determined by its strict
downset. This result looks interesting by itself, but it is also
used in the
proof of Theorem \ref{th:aut_oplus} which states  that
${\mathrm{Aut}}(\mathcal{E}_n, \oplus) \simeq
{\mathcal{S}}_n\times C_2$. It follows that for $n>2$ we have
$${\rm Aut}(\mathcal{E}_n,\, \odot ) = {\rm Aut}(\mathcal{E}_n,\,
\oplus ) = {\rm Aut} (\mathcal{E}_n,\, \leq ) =  {\rm
Aut}(\mathcal{E}_n,\, \odot ,\oplus  ,0) \cong  {\mathcal{S}}_n
\times C_2, \;(n>2),$$ which reflects some harmony between   the
various structures on $\mathcal{E}_n $. The latter demonstrates
some kind of a symmetry which exists in  the class of the $n
\times n$-tiled orders over a fixed discrete valuation ring.

\section{Proof of the Structure Theorem}\label{sec:proof_structure}

For $A\in {\mathcal E}_n$ and $k\geq 1$ by $A^{\odot k}$ we denote the matrix $A \odot  \ldots  \odot  A$ where the number of factors $A$ is $k$.  We also define the partial ordering on ${\mathcal E}_n$ by $A\leq B$ if and only if $A\oplus B=B$. This is equivalent to the condition $a_{ij}\leq b_{ij}$ for all $1\leq i,j\leq n$.

We assume that $A\neq 0$. Fix $p\in\{1,\dots, n\}$ such that the $p$th row of $A$ is non-zero. Let
 $C$ be the set of numbers in ${\mathbb N}\cup \{0\}$ which occur in the $p$th row of $A$. We note that $0 \in C$ since $a_{pp}=0$. Let $r$ be the maximal element of $C$. Thus $|C|\geq 2$.

 For each $c\in C$ we  define the following sets of indices:
\begin{equation}\label{eq:ind}
 J_{c}=\{j\colon a_{pj}=c\}.
\end{equation}
 We order the elements of $C$ assuming that
 $$
 C=\{c_1,\dots, c_m\} \text{ where } c_1< c_2 <\ldots < c_{m}.
 $$
Observe that $c_1=0$ and $c_m=r$.

Further, for each $t\in \{1,\dots, m-1\}$ we put
\begin{equation}\label{eq:it}
I_t=J_{c_1}\cup \ldots \cup J_{c_t} \text{ and } k_t=c_{t+1}-c_t
\end{equation}
and
\begin{equation}\label{eq:tpq}
T(p)=T_{I_1}^{\odot k_1}\odot \ldots \odot T_{I_{m-1}}^{\odot k_{m-1}}.
\end{equation}

Let us prove that

\begin{equation}\label{eq:h}
T(p) \leq A.
\end{equation}

Let $i,j\in \{1,2,\dots, n\}$. We need to show that $T(p)_{ij}\leq a_{ij}$.
The construction of the sets $J_t$ implies that
$$J_{c_1}\cup\ldots \cup J_{c_m}=\{1,2,\dots, n\}.$$
Since, in addition, the above union is  disjoint, there are unique sets $J_{c_s}$  and $J_{c_v}$ such that
$i\in J_{c_s}$ and $j\in J_{c_v} $.

For each $t=1,\dots, m-1$ we note that the block $T_{I_t}$ has $1$  precisely at positions with indices $ij$ where $i\in J_{c_1}\cup \ldots \cup J_{c_t}$ and $j\in J_{c_{t+1}}\cup \ldots \cup J_{c_m}$. It follows that if $s\geq v$  then $T(p)_{ij}=0\leq a_{ij}$. Assume now that $s<v$. Then the matrices $T_{I_s}, T_{I_{s+1}},\dots, T_{I_{v-1}}$ have $1$ at $i,j$ position and all the other matrices $T_{I_t}$ have $0$ at the same position. It follows that
$$
T(p)_{ij}=k_s+\ldots +k_{v-1}=c_v-c_s.
$$
It remains to show that $a_{ij}\geq c_v-c_s$. By  condition (R2) in the definition of ${\mathcal E}_n$ we have the inequality
\begin{equation}\label{eq:a1}
a_{pi}+a_{ij}\geq a_{pj}.
\end{equation}
From $i\in J_{c_s}$ and $j\in J_{c_v}$ we have that $a_{pi}=c_s$ and $a_{pj}=c_v$ by  \eqref{eq:ind}. This and the inequality \eqref{eq:a1} yield that $c_s+a_{ij}\geq c_v$ so that $a_{ij}\geq c_v-c_s$, as required.

We now show that the $p$th row of $T(p)$ equals the $p$th row of $A$:
\begin{equation}\label{eq:h1}
T(p)_{pj}=a_{pj} \text{ for all } j\in\{1,\dots, n\}.
\end{equation}
Indeed, let $j\in\{1,\dots, n\}$. Assume that $j\in J_{c_t}$. Notice that $p \in J_1 = I_1 \subset I_2 \subset \ldots,$ as $a_{pp}=c_1=0.$  From the construction of the matrix $T(p)$ we have
$$
T(p)_{pj}=k_1+\ldots + k_{t-1 }= c_t-c_0=c_t.
$$
But $j\in J_{c_t}$ is equivalent to $a_{pj}=c_t$, so that \eqref{eq:h1} follows.

From \eqref{eq:h} and  \eqref{eq:h1} we immediately obtain

\begin{equation}\label{eq:rcd}
A=\bigoplus \{T(p)\colon p\text{th row of } A \text{ is non-zero}\},
\end{equation}
which finishes the proof of the fact that the matrices $T_I$ form a basis of the algebra $({\mathcal{E}}_n, \odot, \oplus, 0)$.

We are left to prove the claim about minimality. Let $A$ be a non-zero matrix from ${\mathcal E}_n$  and assume  that the index $i$ is such that $a_{ij}\neq 0$ for some $j$. Let $C_i=\{j\colon a_{ij}= 0\}$. Thus $j\not\in C_i$. We show that $T_{C_i}\leq A$.
If $C_i=\{i\}$, then $T_{\{i\}}\leq A$. Otherwise, let $t\in C_i$, $t\neq i$ and let $k\in C_i^c$. Since $a_{it}+a_{tk}\geq  a_{ik}\geq 1$ and since $a_{it}=0$, it follows that $a_{tk}\geq 1$. This implies $T_{C_i}\leq A$, as desired. So for any non-zero $A\in {\mathcal E}_n$ there is some block matrix, which is less then or equal to $A$. The statement about of the minimality of the basis of block matrices now follows from the fact that any two block matrices are incomparable with respect to $\leq$.

\begin{remark}\label{rem:zero-one}
{\em Assume that all elements of the matrix $A$ are zeros and
ones. Then for each $p$ such that the $p$th row of $A$ is non-zero
we have that the matrix $T(p)$, as in  \eqref{eq:tpq}, equals
$T_{I_1}$ (since $m=2$ and $k_1=1$). It follows that the row
decomposition \eqref{eq:rcd} of $A$ in this case does not involve
the operation $\odot$, and thus $A$ is an $\oplus$-combination of
matrices from~${\mathcal T}$.}
\end{remark}

\begin{remark}
{\em The construction of the matrix $T(p)$ is carried over as
follows. The set $I_1$  is the smallest subset of $\{1,\dots, n\}$
such that the $p$th row (and thus any row)  of $T_{I_1}$ is
less than or equal to  the $p$th row of $A$ and $k_1$
is the maximal power of $T_{I_1}$ such that $T_{I_1}^{\odot
k_1}\leq A$. Then to construct $T_{I_2}$ we find the smallest
subset $I_2$ of $\{1,\dots, n\}$ such that the $p$th row of
$T_{I_1}^{\odot k_1}\odot T_{I_2}$ is less  than or
equal to the $p$th row of $A$ and we let $k_2$ be the greatest power
of $T_{I_2}$  such that the $p$th row of $T_{I_1}^{\odot k_1}\odot
T_{I_2}^{\odot k_2}$ is less than or equal to  the
$p$th row of $A$. We construct the subsequent blocks
$T_{I_t}$ and their powers $k_t$ similarly.}
\end{remark}

We emphasize that  not only we have  proved our theorem but also we have suggested an explicit construction of a decomposition of the form \eqref{eq:anja} which has no more than $n$ summands for every  matrix $A\in {\mathcal E}_n$.

We provide an example of the calculation of the matrix $T(p)$.

\begin{example}
{\em  Let $n=9$ and let $A\in {\mathcal E}_9$ be a matrix whose
first row equals}

$$
\begin{pmatrix}
0 & 5 & 0 & 0 & 1 & 3 & 3 & 3 & 5
\end{pmatrix}.
$$
{\em We construct the matrix $T(1)$. Firstly, we have that
$C=\{0,1,3,5\}$ is the set of all elements which occur in the
given row. Now, we calculate the sets $J_c$ for all  $c\in C$:}
$$
J_0=\{1,3,4\}, \, J_1=\{5\}, J_3=\{6,7,8\}, J_5=\{2,9\}.
$$

{\em Further, for each $t=1,\dots, |C|-1=3$ we define the set
$I_t$ and the number $k_t$ according to \eqref{eq:it}:}
$$
I_1=J_0=\{1,3,4\}, \,\, I_2=J_0\cup J_1=\{1,3,4,5\}, \,\, I_3=J_0\cup J_1\cup J_3=\{1,3,4,5,6,7,8\};
$$
$$
k_1=1-0=1, \,\, k_2=3-1=2, \,\, k_3=5-3=2.
$$
{\em Following \eqref{eq:tpq}, we obtain}
$$
T(1)=T_{\{1,3,4\}}\odot T_{\{1,3,4,5\}}^{\odot 2}\odot T_{\{1,3,4,5,6,7,8\}}^{\odot 2}.
$$
\end{example}

We call the decomposition \eqref{eq:rcd}  the {\em row decomposition of} the matrix $A$.
We now introduce the notion of the column decomposition of the matrix $A$. Let
$$
A^t=\bigoplus \{T(p)\colon p\text{th column of } A \text{ is non-zero}\},
$$
be the row decomposition of the transpose $A^t$ of the matrix $A$.

Since clearly the operation $\oplus$ commutes with taking the transpose, transposing the latter equality we obtain
$$
A=\bigoplus \{T(p)^t\colon p\text{th column of } A \text{ is non-zero}\}.
$$
Note that for a block $T_I$ its transpose $T_I^t$ is the block $T_{I^c}$, where by $I^c$ we denote the complement
$\{1,\dots, n\}\setminus I$. Since the operation  $\odot$  also commutes with taking the transpose, we can readily calculate the transpose of each summand $T(p)$. If
$$T(p)=T_{I_1}^{\odot k_1}\odot \ldots \odot T_{I_{m-1}}^{\odot k_{m-1}}$$
then we put
$$
S(p)=T(p)^t=T_{I_1^c}^{\odot k_1}\odot \ldots \odot T_{I_{m-1}^c}^{\odot k_{m-1}}.
$$
We call the decomposition
\begin{equation}\label{eq:ccd}
A=\bigoplus \{S(p)\colon p\text{th column of } A \text{ is non-zero}\}.
\end{equation}
the {\em column decomposition} of $A$.

The technique we have developed so far may be effectively used to verify if a given $n\times n$ matrix over ${\mathbb N}\cup\{0\}$ belongs to ${\mathcal E}_n$. Firstly, for any such  a matrix we can calculate the matrices $T(p)$ and $S(p)$ using our constructions.

\begin{proposition} Let $A$ be an $n\times n$ matrix over ${\mathbb N}\cup\{0\}$. The following statements are equivalent:
\begin{enumerate}
\item $A\in {\mathcal E}_n$.
\item $T(p)\leq A$ for every non-zero row of $A$.
\item $S(p)\leq A$ for every non-zero column of $A$.
\end{enumerate}
\end{proposition}

\begin{proof} The implication (1) $\Rightarrow$ (2) was shown in the proof of Theorem \ref{th:structure_th}.
For the converse implication, we observe that the corresponding part of the proof of Theorem \ref{th:structure_th} shows that  also $T(p)\leq A$ implies that that
all the inequalities (R2)  of the form $a_{pi}+a_{ij}\geq a_{pj}$ hold.
The equivalence (1) $\Leftrightarrow (3)$ follows from (1) $\Leftrightarrow
(2)$ and the observation that $A\in {\mathcal E}_n$ implies that $A^t\in {\mathcal E}_n$. The remaining equivalence now also follows.
\end{proof}

We now provide an example of the calculation of the row decomposition and the column decomposition of a matrix $A\in {\mathcal E}_n$.

\begin{example}{\em Let}
$
A=\left(\begin{array}{cccc}0 & 2 & 5 & 5 \\
4 & 0 & 3 & 3 \\
6 & 2 & 0 & 2 \\
4 & 4 & 2 & 0\end{array}\right). $ {\em The row decomposition of
$A$ is}
\begin{multline*}
A= T(1)\oplus T(2)\oplus T(3)\oplus T(4)  =\\
(T_{\{1\}}^{\odot 2}\odot T_{\{1,2\}}^{\odot 3}) \oplus (T_{\{2\}}^{\odot 3}\odot T_{\{2,3,4\}}) \oplus
 (T_{\{3\}}^{\odot 2}\odot T_{ \{2,3,4\}}^{\odot 4}) \oplus (T_{\{4\}}^{\odot 2}\odot T_{\{3,4\}}^{\odot 2}).
\end{multline*}

{\em The column decomposition of $A$ is}
\begin{multline*}
A= S(1)\oplus S(2)\oplus S(3)\oplus S(4) =\\
(T_{\{2,3,4\}}^{\odot 4}\odot T_{\{3\}}^{\odot 2}) \oplus (T_{\{1,3,4\}}^{\odot 2}\odot T_{\{4\}}^{\odot 2}) \oplus
 (T_{\{1,2,4\}}^{\odot 2}\odot T_{\{1,2\}}\odot T_{\{1\}}^{\odot 2}) \oplus (T_{\{1,2,3\}}^{\odot 2}\odot T_{\{1,2\}}\odot T_{\{1\}}^{\odot 2}).
\end{multline*}
{\em From $T(1),T(2),T(3),T(4)\leq A$ we see that we indeed have
$A\in {\mathcal E}_4$.}
\end{example}

\section{Automorphisms of $({\mathcal E}_n, \odot)$}\label{sec:aut_odot}

  In this section, we study automorphisms of the semigroup $({\mathcal E}_n, \odot)=({\mathcal E}_n, +)$. We denote by $e_{ij}$ the matrix whose entry at  position $i,j$ equals $1$ and all the other entries are $0$'s.

    We begin by observing that
if $A\in {\mathcal E}_n$ then $A^t \in {\mathcal E}_n$, too. We thus have an action of the two-element group $C_2=\{e,a\}$ on $({\mathcal E}_n, \oplus, \odot, 0)$ where $e$ is the identity map, and $a$ acts by $a\cdot A= A^t$. Furthermore, let $\sigma\in {\mathcal{S}}_n$ and $A=(a_{ij})\in {\mathcal E}_n$. We put $\sigma\cdot A =  (a_{\sigma(i)\sigma(j)})$. We observe that $\sigma\cdot A\in {\mathcal E}_n$ and that we have an action of ${\mathcal{S}}_n$ on $({\mathcal E}_n, \oplus, \odot, 0)$. It is clear that this action commutes with the action of $C_2$,  and we obtain an action by automorphisms of the group $C_2\times {\mathcal{S}}_n$ on $({\mathcal E}_n, \oplus, \odot, 0)$.   This action is   faithful if $n>2,$ since  $\sigma \cdot  T_{\{i\}} =  T_{\{\sigma (i)\}} \neq T_{\{i\}^c}=T_{\{i\}}^t,$ for any $\sigma \in {\mathcal S}_n$ and $i=1, \ldots, n.$ As to the case $n=2,$ the action of the unique nontrivial permutation in ${\mathcal{S}}_2$ coincides with the transpose.

In the case $n=2$ we easily have that $ {\rm Aut}(\mathcal{E}_2) = C_2.$ Indeed, observe that  any non-negative $2\times 2$-matrix  whose diagonal entries are $0$'s is exponent   and   the unique minimal generating set of $\mathcal{E}_2 =(\mathcal{E}_2,\, \odot)$ is $\{ e_{12},\, e_{21}\}$. Then any automorphism of $\mathcal{E}_2$ preserves  $\{ e_{12},\, e_{21}\}$, and consequently, a non-trivial automorphism maps $e_{12} \mapsto e_{21}$ and $e_{21}\mapsto  e_{12}.$

For $n>2$ we prove that any automorphism of $({\mathcal E}_n, \odot)$ belongs to $C_2\times {\mathcal{S}}_n$:

\begin{theorem}\label{th:aut_odot} Let $\varphi$ be an automorphism of $({\mathcal E}_n, \odot)$ where $n>2$. Then $\varphi \in C_2\times {\mathcal{S}}_n$.
\end{theorem}

So assume for the rest of the section that $n>2$ and let $\varphi\colon A\mapsto \varphi(A)$ be an automorphism of  $({\mathcal E}_n, \odot)$.
Let us introduce some notation.
We put
$$
{\mathcal L}=\{T_{\{i\}}\colon 1\leq i\leq n\}, \,\, {\mathcal C}=\{T_{\{i\}^c}\colon 1\leq i\leq n\}.
$$
For each ordered pair $(i,j)$ where $i\neq j$ and $i,j\in \{1,\dots, n\}$ we put
$$
A_{ij}=T_{\{i\}}\oplus T_{\{j\}^c},\,\,
L_{ij}= T_{\{i,j\}}\oplus T_{\{j\}}, \,\, C_{ij }=T_{\{i,j\}^c}\oplus T_{\{i\}^c}.
$$
Let, further,
${\mathcal A}$, ${\tilde{\mathcal L}}$ and ${\tilde{\mathcal C}}$ be the sets consisting  of all matrices $A_{ij}$, $L_{ij}$ and $C_{ij}$, respectively.

We say that $A\in {\mathcal E}_n$ is $\odot$-irreducible, if $A$ can not be decomposed as $A=B\odot C=B+C$ where $B,C\neq A$.

\begin{lemma}\label{lem:irreducible} The matrices $A_{ij}$, $L_{ij}$ and $C_{ij}$ are $\odot$-irreducible for all ordered pairs $(i, j)$  where $i\neq j$ and $i,j\in \{1,\dots, n\}$.
\end{lemma}

\begin{proof} We prove the claim for $i=1$ and $j=2$, the general case follows applying some $\sigma\in {\mathcal{S}}_n$ satisfying $\sigma(1)=i$ and $\sigma(2)=j$. Assume that $A_{12}=B\odot C$ where $B,C\in {\mathcal E}_n$, $B,C\neq A$. By Theorem \ref{th:structure_th} both $B$ and $C$ can be decomposed as $(\odot, \oplus)$-expressions in blocks. Since, clearly, all the blocks $T$ in this decomposition satisfy $T\leq B,C\leq A_{12}$, it follows that we must have $T=T_{\{1\}}$ or $T=T_{\{2\}^c}$. Furthermore, both $T_{\{1\}}$ and $T_{\{2\}^c}$ must appear in the decompositions of $B$ and $C$ (at least once in the two decompositions). Also, if $T_{\{1\}}$ does not appear in one of the decompositions, say of $B$, then it must appear in the other decomposition, as otherwise $A_{ij}=B\odot C$ can not hold. But then, applying distributivity, it follows that $T_{\{1\}}\otimes T_{\{2\}^c}$ is a summand of a decomposition of $A_{12}$ which is impossible because the $1,2$-entry of $T_{\{1\}}\otimes T_{\{2\}^c}$ equals $2$, while the $1,2$-entry of $A_{12}$ is~$1$.
\end{proof}

We now analyze the sums of entries of the matrices $T_I\in {\mathcal T}$. We have that the some of entries of $T_I$ equals the number of positions with $1$. The latter number equals $|I| (n-|I|)$. It is easy to see that $1\cdot (n-1)\leq 2(n-2)\leq \ldots \leq \lfloor n/2 \rfloor (n-\lfloor n/2 \rfloor)$, moreover the number of non-zero entries of $T_I$ and $T_{I^c}$ are the same. In particular, we have:

\begin{lemma} The set of non-zero matrices in ${\mathcal E}_n$ with minimal sum of entries is ${\mathcal L}\cup {\mathcal C}$.
\end{lemma}

For $A\in {\mathcal E}_n$ let $\#(A)$ denote the sum of all entries of $A$. We set $U\in {\mathcal E}_n$ to be the matrix whose diagonal entries are $0$'s, and all other entries are $1$'s, that is,
$$U=\sum_{i=1}^n T_{\{i\}}=\sum_{i=1}^n T_{\{i\}^c}.$$ This matrix will play an important role in our considerations.
Furthermore, let ${\mathcal G}\subseteq {\mathcal E}_n$ denote the set of $\otimes$-irreducible matrices. Clearly, ${\mathcal T}\subseteq {\mathcal G}$. Also, by Lemma \ref{lem:irreducible}, ${\mathcal A},{\tilde{\mathcal L}},{\tilde{\mathcal C}}\subseteq  {\mathcal G}$.

Let $\varphi$ be an automorphism of $({\mathcal E}_n, \odot)$. Clearly $\varphi({\mathcal G})={\mathcal G}$. The following important step is in observing that $\varphi$ fixes $U$ and that $\#(A)$ is invariant under the action of $\varphi$ for any $A\in {\mathcal G}$.

\begin{lemma}\label{lem:important}\mbox{}
\begin{enumerate}
\item $\sum \{A\colon A\in {\mathcal G}\}=tU$ for some $t\in {\mathbb N}$.
\item $\varphi(U)=U$.
\item $\#(\varphi(A)) =\#(A)$.
\item $\varphi({\mathcal L}\cup {\mathcal C})={\mathcal L}\cup {\mathcal C}$.
\item $\varphi({\mathcal L})={\mathcal L}$ or $\varphi({\mathcal L})={\mathcal C}$ (and thus, respectively, $\varphi({\mathcal C})={\mathcal C}$ or $\varphi({\mathcal C})={\mathcal L}$).
\end{enumerate}
\end{lemma}
\begin{proof} (1) We begin by observing that ${\mathcal G}$ is invariant under the action of  ${\mathcal S}_n$. It follows that ${\mathcal G}$ is a union of several ${\mathcal S}_n$-orbits. If $A=(a_{ij})\in {\mathcal G}$ then the entry at the position $pq$ with $p\neq q$ of the sum of all elements of the orbit of $A$ equals $\sum_{\sigma\in {\mathcal S}_n}a_{\sigma^{-1}(p)\sigma^{-1}(q)}$. It follows that this entry equals $(n-2)! \#(A)$. Thus the sum of all the matrices in the orbit of $A$ equals $(n-2)!\#(A)U$, and (1) follows.

(2) follows form (1) as $\varphi({\mathcal G})={\mathcal G}$.

(3) For each $i\geq 1$ let $M_i=\{A\in {\mathcal G}\colon \#(A)=i\}$. Let $i_0$ be the minimal $i$ for which $M_i\neq \varnothing$ (it is easy to see that $i_0=n-1$). Let $A\in M_i$. As above, the sum of all matrices of the orbit of $A$ under the action of ${\mathcal S}_n$ equals $(n-2)!i_0U$, and, because the latter matrix is fixed by $\varphi$, we have the sum of all matrices of the orbit of $A$ is fixed by $\varphi$, too:
$\sum_{\sigma\in {\mathcal S}_n}\varphi(\sigma(A))=\sum_{\sigma\in {\mathcal S}_n}\sigma(A)$. On both sides of this equality we have a sum of $n!$ matrices. The sum in the right-hand side is such that the sum of entries of each its member is minimal possible, $i_0$. It follows that the same must be true about the sum in the left-hand side, which, too, has $n!$ summands. We thus have that $\#(\varphi(\sigma(A))=i_0$ for all $\sigma\in {\mathcal S}_n$. This implies that $\varphi(M_{i_0})=M_{i_0}$.

Let $i_1$ be the minimal $i>i_0$ such that $M_{i_1}\neq \varnothing$. Repeating the argument above and using the fact that $\varphi(M_{i_0})=M_{i_0}$ we obtain that $\varphi(M_{i_1})=M_{i_1}$. Aplying induction, it follows that
$\varphi(M_{i})=M_{i}$ for all $i\geq 1$.

(4) As was mentioned in the proof of (3) above, $i_0=n-1$ and we have $M_{i_0}={\mathcal L}\cup {\mathcal C}$. Thus
$\varphi({\mathcal L}\cup {\mathcal C})={\mathcal L}\cup {\mathcal C}$, as desired.

(5) By the above we have that, for each $i=1,\dots, n$,  $\varphi(T_{\{i\}})$ equals either some $T_{\{j\}}$, or some $T_{\{j\}^c}$. Applying $\varphi$ to the equality $\sum_{i=1}^nT_{\{i\}}=U$, we obtain
$$
T_{\{i_1\}}+\ldots + T_{\{i_p\}} + T_{\{i_{p+1}\}^c}+\ldots + T_{\{i_n\}^c}=U,
$$
But, unless $p=0$ or $p=n$, the sum in the left-hand side has at least one entry which is greater than one which is a contradiction. The statement follows.
\end{proof}

We now turn to the behavior of the image of the set ${\mathcal A}\cup {\tilde{\mathcal L}}\cup {\tilde{\mathcal C}}$ with respect to $\varphi$. We first observe that $\#(A)=2n-3$ for $A\in {\mathcal A}\cup {\tilde{\mathcal L}}\cup {\tilde{\mathcal C}}$. Note also that for any such $A$ all its entries are $0$'s and $1$'s: indeed, assuming that  $a_{ij}\geq 2$  we would have $a_{ik}+a_{kj}\geq a_{ij}\geq 2$ for all $k=1,\dots, n$, yielding $\#(A)\geq 2n$, a contradiction.

\begin{lemma}\label{lem:further} $\varphi({\mathcal A}\cup {\tilde{\mathcal L}}\cup {\tilde{\mathcal C}})={\mathcal A}\cup {\tilde{\mathcal L}}\cup {\tilde{\mathcal C}}$.
\end{lemma}

\begin{proof}
We divide the proof into several cases.

{\bf Case 1.} Assume first that $n>6$.  Let $A\in {\mathcal A}\cup {\tilde{\mathcal L}}\cup {\tilde{\mathcal C}}$.
We have $\#(\varphi(A))=2n-3$ by part (3) of Lemma \ref{lem:important}. As all the entries of $\varphi(A)$ are $0$'s and $1$'s, Remark~\ref{rem:zero-one} implies that $\varphi(A)$ can be expressed as a $\oplus$-combination of blocks. Any matrix $B$ involved into such a combination combination must satisfy $\#(B)\leq 2n-3$. Thus
we must have $B=T_I$ or $B=T_{I^c}$ with $|I|=1$ or $|I|=2$ (because if $|I|\geq 3$ we would have $\#(B)\geq 3(n-3)=3n-9>2n-3$ as $n>6$). Note that $\#(B)=n-1$ if $|I|=1$ and $\#(B)=2n-4$ if $|I|=2$. It easily follows that for pairwise distinct $B_1, B_2, B_3$ equal to $T_I$ or $T_{I^c}$ with $|I|=1$ or $|I|=2$ we have $\#(B_1\oplus B_2 \oplus B_3)>2n-3$, so we are left to consider only the case where $\varphi(A)=B_1\oplus B_2$ with $B_1, B_2$ of the form $T_I$ or $T_{I^c}$ where $|I|=1$ or $|I|=2$. This and $\#(\varphi(A))=2n-3$ yield that $\varphi(A)\in {\mathcal A}\cup {\tilde{\mathcal L}}\cup {\tilde{\mathcal C}}$, as desired.

{\bf Case 2.} Assume that $n=6$. In this case $2n-3=9$. We first prove that $\#(A)=9$ implies that $A\in T_I$ where $|I|=3$, or  $A\in {\mathcal A}\cup {\tilde{\mathcal L}}\cup {\tilde{\mathcal C}}$. As all entries of $A$ are zeros and ones, we have that $A$ is expressed as a $\oplus$-combination of matrices $T_I$.  If such a combination contains $B= T_I$ with $|I|=3$ then $\#(B)=9$ and we must have $A=B$. Assume that such a combination contains only matrices of the form $T_I$ or $T_{I^c}$ where $|I|=1$ or $|I|=2$. Then $\#(B)=5$ or $\#(B)=8$. A similar analysis as in the previous case leads to $A\in {\mathcal A}\cup {\tilde{\mathcal L}}\cup {\tilde{\mathcal C}}$.

Let ${\mathcal T}_3=\{T_I\colon |I|=3\}$. From the previous paragraph and part (3) of Lemma \ref{lem:important} it follows that $\varphi({\mathcal T}_3\cup {\mathcal A}\cup {\tilde{\mathcal L}}\cup {\tilde{\mathcal C}})={\mathcal T}_3\cup {\mathcal A}\cup {\tilde{\mathcal L}}\cup {\tilde{\mathcal C}}$. The needed equality $\varphi( {\mathcal A}\cup {\tilde{\mathcal L}}\cup {\tilde{\mathcal C}})={\mathcal A}\cup {\tilde{\mathcal L}}\cup {\tilde{\mathcal C}}$ will follow if we prove that $\varphi({\mathcal T}_3)={\mathcal T}_3$. Let $A\in {\mathcal T}_3$. We assume that $A=T_{\{p,q,r\}}$. Consider the equality
$$
T_{\{p,q,r\}}+T_{\{p\}^c}+T_{\{q\}^c}+T_{\{r\}^c}=T_{\{p,q,r\}^c}+T_{\{p\}}+T_{\{q\}}+T_{\{r\}}.
$$
We apply $\varphi$ to both sides of this equality. In view of part (5) of Lemma \ref{lem:important} we obtain either
$$
\varphi(T_{\{p,q,r\}})+T_{\{i_1\}^c}+T_{\{i_2\}^c}+T_{\{i_3\}^c}=\varphi(T_{\{p,q,r\}^c})+T_{\{j_1\}}+T_{\{j_2\}}+T_{\{j_3\}}
$$
or
$$
\varphi(T_{\{p,q,r\}})+T_{\{i_1\}}+T_{\{i_2\}}+T_{\{i_3\}}=\varphi(T_{\{p,q,r\}^c})+T_{\{j_1\}^c}+T_{\{j_2\}^c}+T_{\{j_3\}^c}
$$
where $i_1,i_2,i_3$, as well as $j_1,j_2,j_3$ are pairwise distinct.
Without loss of generality, we assume that we obtain the former equality. The matrix $T_{\{j_1\}}+T_{\{j_2\}}+T_{\{j_3\}}$ has in the rows $j_1, j_2, j_3$ all elements, but the diagonal ones, equal to $1$. It follows that the matrix in the left-hand side, $\varphi(T_{\{p,q,r\}})+T_{\{i_1\}^c}+T_{\{i_2\}^c}+T_{\{i_3\}^c}$, must be greater than or equal to this matrix. It easily follows that $\varphi(T_{\{p,q,r\}})\geq T_{\{j_1,j_2,j_3\}}$ and thus, as $\#(\varphi(T_{\{p,q,r\}}))=9$, we get the equality $\varphi(T_{\{p,q,r\}})=T_{\{j_1,j_2,j_3\}}$. It follows that $\varphi({\mathcal T}_3)={\mathcal T}_3$, as desired.

{\bf Case 3.} Assume that $n=5$. Then $2n-3=7$. If $|I|=1$, $\#(T_{|I|})=\#(T_{|I|^c})=4$, if $|I|=2$, $\#(T_{|I|})=\#(T_{|I|}^c)=6$. Considering $\oplus$-combinations of such matrices, similarly as in Case 1 above, yields that if $A\in {\mathcal A}\cup {\tilde{\mathcal L}}\cup {\tilde{\mathcal C}}$ then
$\varphi(A)\in {\mathcal A}\cup {\tilde{\mathcal L}}\cup {\tilde{\mathcal C}}$, too.

The cases where $n=3$ and $n=4$ can be treated similarly and are left to the reader.
\end{proof}

\begin{lemma}\mbox{}
\begin{enumerate}
\item There is $\sigma\in {\mathcal S}_n$ such that we have either $\varphi\left(T_{\{i\}}\right)=T_{\{\sigma(i)\}}$ and $\varphi\left(T_{\{i\}^c}\right)=T_{\{\sigma(i)\}^c}$, or $\varphi\left(T_{\{i\}}\right)=T_{\{\sigma(i)\}^c}$ and $\varphi\left(T_{\{i\}^c}\right)=T_{\{\sigma(i)\}}$.
\item $\varphi(\mathcal{A})=\mathcal{A}$.
\item If $i\neq j$ then $\varphi(A_{ij})=\varphi\left(T_{\{i\}}\right)\oplus \varphi\left(T_{\{j\}^c}\right)$.
\end{enumerate}
\end{lemma}
\begin{proof}
(1) From part (5) of Lemma \ref{lem:important} we know that $\varphi({\mathcal L})={\mathcal L}$ or $\varphi({\mathcal L})={\mathcal C}$ (and thus, respectively, $\varphi({\mathcal C})={\mathcal C}$ or $\varphi({\mathcal C})={\mathcal L}$). We assume that $\varphi({\mathcal L})={\mathcal L}$ and $\varphi({\mathcal C})={\mathcal C}$, the other case being treated similarly. As $\varphi$ is a bijection, there are $\sigma,\tau\in {\mathcal S}_n$ such that $\varphi\left(T_{\{i\}}\right)=T_{\{\sigma(i)\}}$ and $\varphi\left(T_{\{i\}^c}\right)=T_{\{\tau(i)\}^c}$ for all $i=1,\dots, n$. Hence, we need only to prove that $\tau=\sigma$.

Observe that for each $i\in \{1,\dots, n\}$ the following equality holds:
\begin{equation}\label{eq:aaa}
\sum_{k\neq i} A_{ik} + T_{\{i\}^c} = (n-2)T_{\{i\}} + U.
\end{equation}
Applying $\varphi$ to both sides of this equality, we get
\begin{equation}\label{eq:aab}
\sum_{k\neq i} \varphi(A_{ik}) + T_{\{\tau(i)\}^c} = (n-2)T_{\{\sigma(i)\}} + U,
\end{equation}
or, equivalently,
\begin{equation}\label{eq:aac}
\sum_{k\neq i} \varphi(A_{ik})  = (n-2)T_{\{\sigma(i)\}} + U - T_{\{\tau(i)\}^c},
\end{equation}
Assume, from the converse that $\tau(i)\neq \sigma(i)$. Denote the matrix in the right-hand side of \eqref{eq:aac} by $B=(b_{st})$ and the matrix in the left-hand side by $C=(c_{st})$.
Let $k\neq \tau(i),\sigma(i)$. Then $b_{\sigma(i)k}+b_{k\sigma(i)}=n$. On the other hand, as $\varphi(A_{ik})\in {\mathcal A}\cup {\tilde{\mathcal L}}\cup {\tilde{\mathcal C}}$, for any $p\neq q$ we have that $(\varphi(A_{ik}))_{pq}+(\varphi(A_{ik}))_{qp}\leq 1$. It follows that $c_{pq}\leq n-1$ for all $p\neq q$. Thus the equality $B=C$ can not hold. The obtained contradiction shows that $\tau(i)=\sigma(i)$.

(2) Observe that in \eqref{eq:aac} (we already know that $\tau=\sigma$) the matrix in the right-hand side does not have any zero column. Thus neither does the matrix in the left-hand side. It follows that $\varphi(A_{ik})\in {\mathcal A}\cup {\tilde{\mathcal C}}$. Switching rows and columns, we can write the `transpose' of the equality \eqref{eq:aaa}, then apply $\varphi$ to it and get the `transpose' of \eqref{eq:aac}. We similarly conclude that $\varphi(A_{ik})\in {\mathcal A}\cup {\tilde{\mathcal L}}$. Therefore, $\varphi(A_{ik})\in {\mathcal A}$, as desired.

(3) Observe the $\sigma(i)$th row of the matrix in the right-hand side of \eqref{eq:aac} (we already know that $\tau=\sigma$) has all non-diagonal entries equal to $n-1$. As this is achieved as a sum of $n-1$ matrices of the form $A_{ik}$, we conclude that $\varphi(A_{ik})\in \{A_{\sigma(i)t}\colon t\neq \sigma(i)\}$. Similarly, switching rows and columns, we get $\varphi(A_{ki})\in \{A_{t\sigma(i)}\colon t\neq \sigma(i)\}$. Therefore, $\varphi(A_{ik})=A_{\sigma(i)\sigma(k)}$, and the desired equality follows.\end{proof}

We now conclude the proof of Theorem  \ref{th:aut_odot}. We denote
by $M_n^+$ the additive semigroup of all non-negative integer
$n\times n$-matrices with zero diagonal. For any
$\varphi\in$  ${\mathrm{Aut}}(\mathcal{E}_n, \odot)$  we define an endomorphism
$\widehat{\varphi}$ of $M_n^+$ by $\widehat{\varphi}(e_{ij}) =
\varphi(T_{\{i\}}) + \varphi(T_{\{j\}^c}) -\varphi(A_{ij})$, $i\neq j$. Assume that $\varphi({\mathcal L})={\mathcal L}$.
It follows that there is $\sigma\in {\mathcal S}_n$ such that $$\widehat{\varphi}(e_{ij}) =
\varphi(T_{\{i\}}) + \varphi(T_{\{j\}^c}) -\varphi(A_{ij})=T_{\{\sigma(i)\}}+T_{\{\sigma(j)\}^c}-A_{\sigma(i)\sigma(j)}=e_{\sigma(i)\sigma(j)}.$$

Thus, $\widehat{\varphi}$ is an
automorphism of $M_n^+$ and for $A=(a_{ij})$ we have $\widehat{\varphi}(A) = (a_{\sigma(i)\sigma(j)})$.
It also follows that the restrictions of
$\varphi$ and  $\widehat{\varphi}$ to $\mathcal{L}\cup \mathcal{C}\cup
\mathcal{A}$ coinside.
Moreover, as any $A\in \mathcal{E}_n$ can be written as
$ A =  \sum\limits_{i,j:\, i\neq j}\alpha_{ij}e_{ij} =
\sum\limits_{i,j:\, i\neq j}\alpha_{ij}(T_{\{i\}} + T_{\{j\}^c} - A_{ij})$, the restriction of $\widehat{\varphi}$
to $\mathcal{E}_n$ coinsides with $\varphi$.  The case where $\varphi({\mathcal L})={\mathcal C}$ is similar, as matrix tranposing commutes with the action of ${\mathcal S}_n$.

\section{Automorphisms of $({\mathcal E}_n, \oplus)$}\label{sec:aut_oplus}

\subsection{Strict downsets of elements of ${\mathcal E}_n$}

A {\em strict downset} of $A\in {\mathcal E}_n$, denoted by
$A^{\Downarrow}$ is the set of all $B\in {\mathcal E}_n$
satisfying $B\lneq A$. Clearly, a strict downset of an element
does not allow to reconstruct this element, as, for example, all
minimal elements of ${\mathcal E}_n$, which are the elements of
the set ${\mathcal T}$, have the same strict downset, consisting
of the zero matrix. In this section we prove the following.

\begin{theorem}\label{th:downset} Let $A,B\in {\mathcal E}_n$ and $A\not\in {\mathcal T}$ and assume that $A^{\Downarrow}=B^{\Downarrow}$. Then $A=B$. \end{theorem}

Hence, a non-minimal element of ${\mathcal E}_n$ is uniquely
determined by its strict downset. This result looks interesting on
itself, but we also use it later on for studying automorphisms of
$({\mathcal E}_n, \oplus)$. The remainder of this
subsection will be devoted to the proof of Theorem
\ref{th:downset}.

So assume that $A, B\in {\mathcal E}_n$ are such that $A\not\in
{\mathcal T}$ and that $A^{\Downarrow}=B^{\Downarrow}$. We write
$A=(a_{ij})$ and $B=(b_{ij})$. Let
${\mathrm{Max}}(A^{\Downarrow})$ be the set of maximal elements of
$A^{\Downarrow}$. Clearly, $A^{\Downarrow}=B^{\Downarrow}$ if and
only if
${\mathrm{Max}}(A^{\Downarrow})={\mathrm{Max}}(B^{\Downarrow})$.

Since the $\oplus$ operation on ${\mathcal E}_n$ coincides with
the join with respect to the natural partial order, we have
$$ \oplus\{C \colon C\in {\mathrm{Max}}(A^{\Downarrow})\}\leq A$$
and thus we have that either $\oplus\{C\colon C\in
{\mathrm{Max}}(A^{\Downarrow})\}=A$ or $\oplus\{C\colon C\in
{\mathrm{Max}}(A^{\Downarrow})\}\in
{\mathrm{Max}}(A^{\Downarrow})$. In the latter case we have that
$|{\mathrm{Max}}(A^{\Downarrow})|=1$.

In the former case we have $A=\oplus\{C\colon C\in
{\mathrm{Max}}(A^{\Downarrow})\}=\oplus\{C\colon C\in
{\mathrm{Max}}(B^{\Downarrow})\}=B$. So, to prove Theorem
\ref{th:downset}, we can suppose that
$|{\mathrm{Max}}(A^{\Downarrow})|=|{\mathrm{Max}}(B^{\Downarrow})|=1$.

Consider the column decomposition of $A$. As
$|{\mathrm{Max}}(A^{\Downarrow})|=1$, no $\oplus$-sum of several
matrices which are strictly less than $A$ is equal to $A$. It
follows that at least one $\oplus$-summand in the column
decomposition of $A$ equals $A$. We thus have $S(p)=A$ for some
$p$ (and a similar statement is true for the row decomposition,
but we do not need it here). We will say that $A\in {\mathcal
T}^{\odot m}$ if $A$ can be decomposed as a product of precisely
$m$ $\odot$-factors. If $A=S(p)=T_{I_1}^{\odot k_1}\odot \dots
\odot T_{I_{l}}^{\odot k_{l}}$ with $k_1+\dots + k_l =m$ then
$A\in {\mathcal T}^{\odot m}$ and the maximal element in the $p$th
column of $A$ is $m$ (which is also the maximal element of $A$).

\begin{lemma} \label{lem:down} Let $A,B\in {\mathcal E}_n$ and $A,B\not\in {\mathcal T}$ and assume that $A^{\Downarrow}=B^{\Downarrow}$. Assume also that $|{\mathrm{Max}}(A^{\Downarrow})|=|{\mathrm{Max}}(B^{\Downarrow})|=1$.
\begin{enumerate}
\item If $a_{ij}>b_{ij}$ then $a_{ij}=b_{ij}+1$. \item
$a_{ij}>b_{ij}$ implies that $a_{ij}=m$ and $b_{ij}=m-1$.
\item $a_{ij}=0$ if and only if $b_{ij}=0$.
\end{enumerate}
\end{lemma}

\begin{proof} (1) Assume $a_{ij}\geq b_{ij}+2$. Assume that $A=S(p)=T_{I_1}^{\odot k_1}\odot \dots \odot T_{I_{l}}^{\odot k_{l}}$, where  $I_1\supset \dots \supset I_{l}$ (the inclusions are strict) and $k_1+\dots + k_l\geq 2$ as $A$ is not minimal. Let $I_t$ be such that $i\in I_t$ and consider the matrix obtained from $S(p)$ by removing one factor $I_t$. The obtained matrix belongs to ${\mathcal E}_n$, and its $ij$-entry equals $a_{ij}-1$. Hence, the obtained matrix is strictly less than $A$ and is not less than $B$ which contradicts the assumption that $A^{\Downarrow}=B^{\Downarrow}$. It follows that $a_{ij}=b_{ij}+1$, as required.

(2) Assume $a_{ij}>b_{ij}$ but that $a_{ij}<m$. This implies that
there is some $t$ such that $i\not\in I_t$. Consider the matrix
obtained from $A=S(p)$ by removing the factor $T_{I_t}^{\odot
k_t}$. This matrix is strictly less than $A$ but its  $ij$-entry
 equals $a_{ij}>b_{ij}$.
Thus $A-T_{I_t}^{\odot k_t}$ 
belongs to $A^{\Downarrow}\setminus B^{\Downarrow},$ which
contradicts our assumption.

(3) Assume $b_{ij}=0$ and $a_{ij}\neq 0$. Then $a_{ij}>b_{ij}$
which, by (2) above, means $a_{ij}=1=m$. This is a contradiction
with $m>1$.
\end{proof}

We now complete the proof of Theorem \ref{th:downset}. Let $A,B\in
{\mathcal E}_n$, $A\neq B$ and $A,B\not\in {\mathcal T}$ and
assume that $A^{\Downarrow}=B^{\Downarrow}$. (Note that if $A\in
{\mathcal T}$ and $B\not\in {\mathcal T}$ the equality
$A^{\Downarrow}=B^{\Downarrow}$ can not hold.) We can suppose that
$|{\mathrm{Max}}(A^{\Downarrow})|=|{\mathrm{Max}}(B^{\Downarrow})|=1$.
Let $A=S(j)=T_{I_1}^{\odot k_1}\odot \dots \odot T_{I_{l}}^{\odot
k_{l}}$ with $k_1+\dots + k_l =m$ and let
 $B=S(q)=T_{J_1}^{\odot r_1}\odot \dots \odot T_{J_{s}}^{\odot r_{s}}$ with $r_1+\dots + r_s =l$ be the column decompositions of $A$ and $B$. As $B\not\geq A$ there is an index $ij$ such that $a_{ij}>b_{ij}$. By Lemma \ref{lem:down} we have $a_{ij}=m$ and $b_{ij}=m-1$. This equality and $J_1\supset \dots \supset J_{s}$ mean that $i\in J_1,\dots, J_{s-1}$ and $i\not\in J_s$ (also $r_s=1$). Let $p\in J_s$. This and $i\in J_s^c$ implies that  $b_{pi}\neq 0$. On the other hand, as $a_{ij}=m$, it follows that $i\in I_1,\dots, I_l$, so $i\not\in I_t^c$ for all $t=1,\dots,l$. It follows that $a_{ti}=0$ for all $t=1,\dots, n$. In particular $a_{pi}=0$. This, together with $b_{pi}\neq 0$, contradicts part (3) of Lemma \ref{lem:down}. This finishes the proof.

\subsection{Automorphisms of $({\mathcal E}_n, \oplus)$} Throughout this section, if not stated otherwise, we assume that $n\geq 3$. In this subsection we prove that the automorphisms of the semigroup $({\mathcal E}_n, \oplus)$ are the same as those of the semigroup $({\mathcal E}_n, \odot)$ (cf. Theorem \ref{th:aut_odot}).

\begin{theorem}\label{th:aut_oplus} Let $\varphi$ be an automorphism of $({\mathcal E}_n, \oplus)$ where $n>2$. Then $\varphi \in C_2\times {\mathcal{S}}_n$.
\end{theorem}

For $A,B\in {\mathcal E}_n$ we have $A\leq B$ if and only if
$A\oplus B=B$. Therefore, for $\varphi \in
{\mathrm{Aut}}({\mathcal E}_n, \oplus)$ we have that $A\leq B$ if
and only if $\varphi(A)\leq \varphi(B)$. In other words, an
automorphism of $({\mathcal E}_n, \oplus)$ is an
order-automorphism of $({\mathcal E}_n, \leq)$. And conversely,
since $A\oplus B$ is the join of $A$ and $B$ with respect to
$\leq$, it follows that any order-automorphism of ${(\mathcal
E}_n,\leq)$ is an automorphism of $({\mathcal E}_n, \oplus)$. This
observation will be important in what follows and will be used
without further mention.

Again  the case $n=2$ is very easy. Indeed, since
$\varphi\in {\rm Aut}(\mathcal{E}_2)$ preserves the partial order,
it preserves the set of minimal matrices $\{ e_{12},\, e_{21}\}$,
and consequently, $\varphi$ is either the identity map, or the
transposition of matrices.

Throughout this section, if not stated otherwise, we assume that
$n\geq 3$.

\begin{lemma}\label{lem:tau} Let $\varphi, \psi\in {\mathrm{Aut}}({\mathcal E}_n, \oplus)$ are such that $\varphi(T)=\psi(T)$ for all $T\in {\mathcal T}$. Then $\varphi=\psi$.
\end{lemma}

\begin{proof} For $A\in {\mathcal E}_n$ let $h(A)$ be the biggest integer $k$ such that there exist $A_1,\dots, A_k\in {\mathcal E}_n$ such that $A_1\lneq A_2 \lneq \dots \lneq A_k=A$. Applying $\varphi$, it follows that $h(\varphi(A))\geq h(A)$. As $A=\varphi^{-1}\varphi(A)$, we similarly obtain the opposite inequality, whence $h(\varphi(A))=h(A)$.

We prove the statement of the lemma by induction on $h(A)$. Notice
that $h(A)=1$ if and only if $A\in {\mathcal T}$ and the equality
$\varphi(A)=\psi(A)$ for $A\in {\mathcal T}$ holds by the
assumption of the lemma. We assume that $\varphi(A)=\psi(A)$ for
any $A$ with $h(A)\leq t$, where $t\geq 1$, and prove that
$\varphi(A)=\psi(A)$ for any $A$ with $h(A)=t+1$. Assume that
$\varphi(A)\neq \psi(A)$. Since $h(\varphi(A))=h(\psi(A))=t+1$, it
follows from Theorem \ref{th:downset} that
$\varphi(A)^{\Downarrow}\neq \psi(A)^{\Downarrow}$. Thus, without
loss of generality, there is $C\in {\mathcal E}_n$ satisfying
$C\lneq \varphi(A)$ but $C\not\leq \psi(A)$. Hence,
$\varphi^{-1}(C)\lneq A$ and $\varphi^{-1}(C)\not\leq
\varphi^{-1}\psi(A)$. Observe that $h(\varphi^{-1}(C))\leq t$ and
thus, by the inductive assumption, we have
$C=\varphi\varphi^{-1}(C)=\psi\varphi^{-1}(C)$. Hence,
$\psi^{-1}(C)=\varphi^{-1}(C)$ whence $\psi^{-1}(C) \lneq A$, and
applying $\psi$, we get $C\lneq \psi(A)$, which contradicts our
assumption on $C$ that $C\not\leq \psi(A)$. Therefore, we have
proved that $\varphi(A)=\psi(A)$.
\end{proof}

Let $I_1,I_2\subseteq \{1,\dots, n\}$,  $|I_1|,|I_2|, \in
\{1,\dots, n-1\}$. If the inequality
 \begin{equation}\label{eq:solution}
 T_J\leq T_{I_1}\oplus T_{I_2}
 \end{equation} holds for some $J\subseteq \{1,\dots, n\}$ with $|J|\in \{1,\dots, n-1\}$ we say that
 $T_J$ is a {\em solution} of \eqref{eq:solution} We say that a solution $T_J$ of \eqref{eq:solution} is {\em proper} if
 $J\neq I_1$ and $J\neq I_2$.

 \begin{lemma}\label{lem:solutions} Let $I_1,I_2\subseteq \{1,\dots, n\}$ and $|I_1|,|I_2|, \in \{1,\dots, n-1\}$. Then
 \begin{enumerate}
 \item If $|I_1\cap I_2| \in \{1,\dots, n-1\}$ then $T_{I_1\cap I_2}$ is a solution of \eqref{eq:solution}.
 \item If $|I_1\cup I_2| \in \{1,\dots, n-1\}$ then $T_{I_1\cup I_2}$ is a solution of \eqref{eq:solution}.
\item If  $T_J$ is a proper solution of \eqref{eq:solution} then
either $J=I_1\cap I_2$ of $J=I_1\cup I_2$.
 \end{enumerate}
 \end{lemma}

 \begin{proof} (1) Let $T_{I_1\cap I_2} = (a_{ij})$, $T_{I_1}=(b_{ij})$ and $T_{I_2}=(c_{ij})$.
 (1) We need to show that if $a_{ij}=1$ then $b_{ij}+c_{ij}=1$. We have $a_{ij}=1$ if and only if $i\in I_1\cap I_2$ and $j\in (I_1\cap I_2)^c=I_1^c\cup I_2^c$. If $i\in I_1$ we have $b_{ij}=1$, if $i\in I_2$ we have $c_{ij}=1$, as needed.

 (2) is similar to (1).

 (3)   Let $T_J = (a_{ij})$, $T_{I_1}=(b_{ij})$ and $T_{I_2}=(c_{ij})$. We first show that $I_1\cap I_2\subseteq J$. If $I_1\cap I_2=\varnothing$, we are done. Otherwise, if $I_1\cap I_2\not\subseteq J$, take some $j\in  (I_1\cap I_2) \cap J^c$. Then take some $i\in J$. We have $a_{ij}=1$, but $b_{ij}=c_{ij}=0$, which contradicts $T_J\leq T_{I_1}\oplus T_{I_2}$.

 Assume that $J\cap (I_1\setminus I_2)\neq\varnothing$ and show that $J\supseteq I_1$. If the latter does not hold, we take some $i\in J\cap (I_1\setminus I_2)$ and $j\in I_1\setminus J$. We have $a_{ij}=1$, $b_{ij}=c_{ij}=0$, which again contradicts $T_J\leq T_{I_1}\oplus T_{I_2}$.

 It follows by symmetry that if $J\cap (I_2\setminus I_1)\neq\varnothing$ then $J\supseteq I_2$.

 We finally show that $J\subseteq I_1\cup I_2$. If $I_1\cup I_2=\{1,\dots, n\}$, then we are done. Otherwise, assume that $J\setminus (I_1\cup I_2)\neq \varnothing$ and take some $i\in J\setminus (I_1\cup I_2)=J\cap I_1^c\cap I_2^c$ and any $j\in J^c$. We have $a_{ij}=1$ but $b_{ij}=c_{ij}=0$, a contradiction with $T_J\leq T_{I_1}\oplus T_{I_2}$.
 \end{proof}

 Lemma \ref{lem:solutions} tells us that for any given $I_1,I_2\subseteq \{1,\dots, n\}$ with $|I_1|,|I_2| \in \{1,\dots, n-1\}$, \eqref{eq:solution} has at most two proper solutions. 

 \begin{lemma} \mbox{}
 \begin{enumerate}
 \item We have that $|I_1|\in \{1,n-1\}$ if and only if for any $I_2\subseteq \{1,\dots, n\}$ with $|I_2| \in \{1,\dots, n-1\}$ the inequality \eqref{eq:solution} has at most one proper solution.
 \item Let $\varphi\in {\mathrm{Aut}}({\mathcal E}_n, \oplus)$. Then $\varphi({\mathcal L}\cup {\mathcal C})= {\mathcal L}\cup {\mathcal C}$.
 \item Let $\varphi\in {\mathrm{Aut}}({\mathcal E}_n, \oplus)$. Then either $\varphi({\mathcal L})={\mathcal L}$ or $\varphi({\mathcal L})={\mathcal C}$ (and then, respectively, $\varphi({\mathcal C})={\mathcal C}$ or $\varphi({\mathcal C})={\mathcal L}$).
 \end{enumerate}
 \end{lemma}

 \begin{proof} (1) Assume that $|I_1|=1$. Then we have that either $I_1\cap I_2= \varnothing$, or, otherwise, $I_1\cap I_2=I_1$. Thus $T_{I_1\cap I_2}$ can not be a proper solution of \eqref{eq:solution}.  Assume that $|I_1|=n-1$. Then we have that either $I_1\cup I_2= \{1,\dots, n\}$ or, otherwise,
 $I_2\subseteq I_1$. Thus $T_{I_1\cup I_2}$  can not be  a proper solution of \eqref{eq:solution}.

 Assume now that for any $I_2\subseteq \{1,\dots, n\}$ with $|I_2| \in \{1,\dots, n-1\}$ the inequality \eqref{eq:solution} has at most one proper solution and let us prove that $|I_1|\in \{1,n-1\}$. Assume that $|I_1|\not\in \{1,n-1\}$. Let $x,y,s,t\in \{1,\dots, n\}$ be such that $x\neq y$, $s\neq t$, $x,y\in I_1$ and $s,t\in I_1^c$.  Let $I_2=(I_1\setminus \{x\})\cup \{s\}$. In this case we have $T_{I_1\cup I_2}$ and $T_{I_1\cap I_2}$ are two proper solutions of \eqref{eq:solution}.

(2) follows from (1) because ${\mathcal L}=\{T_I\colon |I|=1\}$
and  ${\mathcal C}=\{T_I\colon |I|=n-1\}$ and an automorphism
preserves the property about proper solutions given in (1).

(3) As $\varphi({\mathcal T})={\mathcal T}$, we have
$\varphi(U)=\varphi(\oplus_{T\in {\mathcal T}} T)= \oplus_{T\in
{\mathcal T}} T=U$. Now, from $\oplus_{i_1}^n T_{\{i\}}=U$, we
have $\oplus_{i=1}^n\varphi (T_{\{i\}})=U$. As $\#(U)=n(n-1)$ and
$\#(A)=n-1$ for any $A\in {\mathcal L}\cup {\mathcal C}$, no two
of the matrices $\varphi (T_{\{i\}})$ above can have overlapping
occurances of $1$. It follows that the set of all $\varphi
(T_{\{i\}})$, $1\leq i\leq n$, is either ${\mathcal L}$ or
${\mathcal C}$.
 \end{proof}

 Let $\varphi\in {\mathrm{Aut}}({\mathcal E}_n, \oplus)$. We assume that $\varphi({\mathcal L})={\mathcal L}$. By the lemma above we have that there are $\sigma,\tau\in {\mathcal S}_n$ such that $\varphi(T_{\{i\}})=T_{\{\sigma(i)\}}$ and $\varphi(T_{\{i\}^c})=T_{\{\tau(i)\}^c}$ for all $i=1,\dots, n$.

We now proceed with the proof of Theorem \ref{th:aut_oplus}. Let
$I\subseteq \{1,\dots, n\}$ be such that $|I|\in\{1,\dots, n-1\}$.
For $S\in {\mathcal T}$ consider the following sets of conditions:
\begin{equation}\label{eq:set1}
S\leq {\oplus_{i\in I} T_{\{i\}}} \text { and } S\not\leq
\oplus_{i\in I\setminus \{j\}} T_{\{i\}} \text{ for any } j\in I;
\end{equation}
\begin{equation}\label{eq:set2}
S\leq \oplus_{i\in I^c} T_{\{i\}^c} \text { and } S\not\leq
\oplus_{i\in I^c\setminus \{j\}} T_{\{i\}^c} \text{ for any } j\in
I^c.
\end{equation}
It is straightforward to verify that there is precisely one $S\in
{\mathcal T}$ which satisfies \eqref{eq:set1}: this is $S=T_I$ and
also there is precisely one $S\in {\mathcal T}$ which satisfies
\eqref{eq:set2}: this is again $S=T_I$.  Applying $\varphi$ to
sets of conditions \eqref{eq:set1} and \eqref{eq:set2}, we obtain

\begin{equation}\label{eq:set3}
\varphi(S)\leq \oplus_{i\in I} T_{\{\sigma(i)\}} \text { and }
S\not\leq \oplus_{i\in I\setminus \{j\}} T_{\{\sigma(i)\}} \text{
for any } j\in I;
\end{equation}
\begin{equation}\label{eq:set4}
\varphi(S)\leq \oplus_{i\in I^c} T_{\{\tau(i)\}^c} \text { and }
\varphi(S)\not\leq \oplus_{i\in I^c\setminus \{j\}}
T_{\{\tau(i)\}^c} \text{ for any } j\in I^c.
\end{equation}

By uniqueness of solution of \eqref{eq:set3} and \eqref{eq:set4}
we obtain that $\varphi(S)=T_{\sigma(I)}=T_{\tau(I)}$. It follows
in particular that $\sigma(I)=\tau(I)$ for any $I\subseteq
\{1,\dots, n\}$ with $|I|\in\{1,\dots, n-1\}$. If we take $|I|=1$
this implies that $\sigma=\tau$.

It follows that if $\varphi\in {\mathrm{Aut}}({\mathcal E}_n,
\oplus)$ and $\varphi({\mathcal L})={\mathcal L}$, there is
$\sigma\in {\mathcal S}_n$ such that for all $A\in {\mathcal T}$
we have $\varphi(T)=\sigma\cdot T$. Now, it follows from Lemma
\ref{lem:tau} that $\varphi(A)=\sigma\cdot A$ for all $A\in
{\mathcal E}_n$. This completes the proof for the case where
$\varphi({\mathcal L})={\mathcal L}$. The case where
$\varphi({\mathcal L})={\mathcal C}$ is considered similarly.

\begin{corollary}\label{aut:max-plus}
 ${\mathrm{Aut}}(\mathcal{E}_n,\, \odot ) = {\mathrm{Aut}}(\mathcal{E}_n,\, \oplus ) = {\mathrm{Aut}}(\mathcal{E}_n,\, \leq ) =  {\mathrm{Aut}}(\mathcal{E}_n,\, \odot ,\oplus  ,0) \cong  S_n\times C_2,$
if $n\geq 3$. \end{corollary}

Notice, that each $n\times n$-tiled order over a discrete
valuation ring $\mathcal{O}$ is conjugate by a matrix from
$\rm{GL}_n(\mathcal{O})$ to a tiled order with non-negative
exponent matrix (see, for example,~\cite{jate1, jate2}).  The set
$\rm{Tiled}(n, \mathcal{O})$ of all $n\times n$ tiles orders over
a fixed $\mathcal{O}$ is a partially ordered set with respect to
the set-theoretic inclusion $\subseteq$, which is anti-isomorphic
to $(\mathcal{E}_n,\, \leq)$. In addition, there is an
anti-isomorphism between $(\rm{Tiled}(n, \mathcal{O}),\, \cap)$
and $(\mathcal{E}_n,\, \oplus)$. Consequently, $${\mathrm{Aut}}
(\rm{Tiled}(n, \mathcal{O}),\, \subseteq) \cong
{\mathrm{Aut}}(\mathcal{E}_n,\, \leq ) \cong
 {\mathrm{Aut}} (\rm{Tiled}(n, \mathcal{O}),\, \cap) \cong {\mathrm{Aut}}(\mathcal{E}_n,\, \oplus ).
$$

\section*{Acknowlegements}
The first named author was partially supported by CNPq of Brazil proc. 305975/2013-7 and partially by Fapesp of Brazil Proc. 2015/09162-9, the third named author was partially supported by Fapesp of Brazil Proc. 2014/23853-1 and partially by ARRS grant P1-0288 (Slovenia), the rest of the authors were supported by Fapesp of Brazil Proc. 2015/16726-6, Proc. 2013/11350-2.


\begin{thebibliography}{10}


\bibitem{AEI} Aceto, L., \'Esik,  Z., Ing\'olfsd\'otter, A., The max-plus algebra of the natural numbers has no finite equational basis, {\it Theoretical Computer Science},
{\bf 293}, (1) (2003), 169--188.

\bibitem{AEI1} Aceto L., \'Esik, Z., Ing\'olfsd\'otter, A., On the two-variable fragment of the equational theory of the max-sum algebra of the natural numbers, in: H. Reichel, S. Tison (eds.), {\it Proc. 17th Internat. Symp. on Theoretical Aspects of Computer Science, Lecture Notes in Computer Science}, {\bf 1770}, Springer, Berlin, 2000, 267--278.

\bibitem{DArnold} Arnold, D. M.,
A finite global Azumaya theorem in additive categories,
{\it Proc. Am. Math. Soc.} {\bf 91}, (1984),  25--30.


\bibitem{DemonetLuo} Demonet, L.,  Luo, X.,
Ice quivers with potential associated with triangulations and Cohen-Macaulay modules over orders, {\it Trans. Am. Math. Soc. }, {\bf 368}, (6), (2016), 4257-4293.

\bibitem{DoKiPlaJAlgAppl}
Dokuchaev, M., Kirichenko, V., Plakhotnyk, M., On  exponent matrices
of tiled orders, {\it J.  Algebra  Appl.}, {\bf 16}, (1), (2016), 1650192.1 - 1650192.25.




 \bibitem{DokKirPolMi} Dokuchaev, M.,  Kirichenko, V.,  Polcino Milies, C.,
 Locally nilpotent groups of units in tiled rings, {\it J. Algebra } {\bf 323},  (11),  (2010),  3055--3066.


\bibitem{DrozdKir1972} Drozd, Yu. A.,  Kirichenko, V. V., On quasi-Bass orders, {\it Izv. Akad. Nauk SSSR Ser. Mat.}, {\bf 36} (2), (1972),  328--370.



\bibitem{DrozdKirRoiter} Drozd, Yu.A., Kirichenko, V.V., Roiter, A.V.,
On heredirary and Bass orders,
{\it Math. USSR, Izv.}  {\bf 31} (6),  (1967), 1415-1436.

\bibitem{EisenRob} Eisenbud, D., Robson, J. C., Hereditary Noetherian prime rings, {\it J. Algebra}, {\bf 16},  (1970), 86-104.

\bibitem{Fujita-1990}
Fujita, H., Tiled orders of finite global dimension, {\it Trans.
Amer. Math. Soc.}  {\bf 322} (1990) 329--342.

\bibitem{Fujita-2001}
Fujita, H., Neat idempotents and tiled orders having large global
dimension, {\it J. Algebra} {\bf 256} No 1 (2002) 194--210.

\bibitem{FujitaSakai} Fujita, H.,  Sakai, Y.,
Frobenius full matrix algebras and Gorenstein tiled orders, {\it Commun. Algebra}, {\bf 34}, (3),  (2006), 1181-1203.





\bibitem{Harada-1963(2)} Harada, M.,
 Structure of hereditary orders over local rings, {\it J. Math. Osaka City Univ.} 1{\bf 4} (1963),
1--22.




\bibitem{jate1} Jategaonkar, V.A., Global dimension of tiled orders over commutative Noetherian domains,
{\it Trans. Am. Math. Soc.} {\bf 190}, (1974),  357--374.


\bibitem{jate2} Jategaonkar, V.A., Global dimension of tiled orders over a discrete valuation ring,
{\it Trans. Am. Math. Soc.}  {\bf 196}, (1974), 313--330.




\bibitem{HGK-1} Hazewinkel, M., Gubareni, N., Kirichenko, V.V., \textit{Algebras, rings and
modules}, Vol. 1, Mathematics and its Applications, 575. Kluwer
Academic Publishers, Dordrecht, 2004.

\bibitem{JanOde-1997}  Jansen, W.S., Odenthal, C. J.,
A tiled order having large global dimension, {\it J. Algebra} {\bf
192} No. 2 (1997) 572--591.

\bibitem{Keating} Keating, M. E., On the $K$-theory of Tiled Orders, {\it J. Algebra}, {\bf 43}, (1976), 193--197.





\bibitem{KirkKuz-1989}
Kirkman, E., Kuzmanovich, J., Global dimensions of a class of
tiled orders, {\it J. Algebra} {\bf 192}, no. 2 (1997) 572--591.



\bibitem{KoenigEtc} Koenig, S.,  Sanada, K.,  Snashall, N., On Hochschild cohomology of orders, {\it Arch. Math.} {\bf  81},  (6),  (2003), 627--635.

\bibitem{MandelSimon}  Mandel, A.,   Simon, I., On Finite Semigroups of Matrices, {\it Theoret. Comput. Sci.} {\bf 5} (1977), 101--111.

\bibitem{PengGuo} Peng, Y.,  Guo, X.,
The K1 group of tiled orders, {\it Commun. Algebra}, {\bf 41}, (10),  (2013), 3739--3744.

\bibitem{PierceEtc} Pierce, R.S., Vinsonhaler, C.,
Realizing tiled orders, {\it J. Algebra}, {\bf 164}, (1), (1994),  26--52.


\bibitem{Pin} Pin, J.-E., Tropical semirings, In {\it Idempotency} (Bristol, 1994), 50--69, Publ. Newton Inst., Vol. 11, Cambridge Univ. Press, Cambridge, 1998.



\bibitem{Reiner} Reiner, I., {\it Maximal orders},
London Mathematical Society Monographs, New Series, 28, Oxford University Press,  2003.



\bibitem{RogKirKhiZhu-2001}
Roggenkamp, K.W., Kirichenko, V.V., Khibina, M.A.,
Zhuravlev, V. N., Gorenstein tiled orders, {\it  Commun. Algebra} {\bf
29}  (9), (2001) 4231--4247.


\bibitem{Rump} Rump, W.,
Enlacements and representation theory of completely reducible orders,
{\it Representation theory II, Groups and orders, Proc. 4th Int. Conf., Ottawa/Can. }1984, Lect. Notes Math. {\bf 1178},  (1986), 272--308.

\bibitem{Rump2} Rump, W.,
Discrete posets, cell complexes, and the global dimension of tiled
orders, {\it Commun. Algebra},  {\bf 24}, (1),  (1996), 55--107.



\bibitem{Simon1978} Simon, I., Limited subsets of a free monoid,
in {\it Proc. 19th Annual Symposium on Foundations of Computer Science},  Piscataway, NJ, 1978, pp. 143-150, IEEE.



\bibitem{Simon1988} Simon, I., Recognizable sets with multiplicities in the tropical semiring, in {\it Mathematical
Foundations of Computer Science} (Carlsbad, 1988),  pp. 107--120, Lecture Notes in Computer Science, Vol. 324,
Berlin, 1988.



\bibitem{Simon1994}  Simon, I.,  On semigroups of matrices over the tropical semiring, {\it RAIRO Inform. Th\'eor. Appl.} {\bf 28}, 3-4 (1994), 277--294.


\bibitem{Simson1} Simson, D.,
A reduction functor, tameness, and Tits form for a class of orders,  {\it J. Algebra}, {\bf 174}, (2),  (1995), 430--452.


\bibitem{Simson2} Simson, D.,
Three-partite subamalgams of tiled orders of finite lattice type, {\it J. Pure Appl. Algebra}, {\bf 138}, (2),  (1999), 151--184.

\bibitem{Simson3} Simson, D.,
A reduced Tits quadratic form and tameness of three-partite
subamalgams of tiled orders, {\it Trans. Am. Math. Soc.},  {\bf 352}, (10),  (2000), 4843--4875.

\bibitem{Tarsy-1970}
 Tarsy, R. B., Global dimension of orders, {\it Trans.  Amer.
Math.Soc.} {\bf  151} (1970) 335--340.



\bibitem{WeideRoggen-1984}
Weidemann, A., Roggenkamp, K.W., Auslander-Reiten quivers of Schurian orders, {\it  Commun.
Algebra},  {\bf 12}, (1984),  2525--2578.



\bibitem{YangYu}  Tse-Chung Yang,  Chia-Fu Yu,
Monomial, Gorenstein and Bass orders, {\it J. Pure Appl. Algebra},  {\bf 219}, (4),  (2015), 767--778.

\bibitem{Zav-Kir} Zavadski, A.G., Kirichenko, V.V.,
Torsion free modules over prime rings, \textit{Zap. Sci. Semin.
LOMI USSR Akad. Sci.}, \textbf{57} (1976), pp. 100--116.

\bibitem{Zav-Kir1977}  Zavadski, A.G., Kirichenko, V.V.,  Semimaximal rings of finite type, {\it Mat. Sb.}, {\bf  103(145)},  3(7),  (1977), 323--345.



\end{thebibliography}
\end{document}